\documentclass[11pt, reqno, oneside]{amsart}
\usepackage{parskip}
\usepackage[]{babel}
\usepackage[utf8]{inputenc} 
\usepackage{amsmath,amsthm,amsbsy}
\usepackage{amsfonts}
\usepackage{mathrsfs}
\usepackage{bbm}
\usepackage[svgnames]{xcolor}
\usepackage{amssymb}
\usepackage[colorlinks=true, menucolor=Blue, linkcolor=Blue, citecolor=Blue]{hyperref}
\usepackage{enumitem}
\usepackage{multicol}
\usepackage{float}
\usepackage{fancyhdr}
\usepackage{layout}
\usepackage{esint}
\usepackage{array}
\usepackage[colorinlistoftodos]{todonotes}
\usepackage{orcidlink} 

\theoremstyle{plain}
\newtheorem{theorem}{Theorem}[section]                                          
                          
\newtheorem{lemma}[theorem]{Lemma}
\newtheorem{example}[theorem]{Example}
\newtheorem{corollary}[theorem]{Corollary}

\theoremstyle{definition}
\newtheorem{definition}[theorem]{Definition}
\theoremstyle{remark}
\newtheorem{remark}[theorem]{Remark}

\makeatletter \@addtoreset{equation}{section} \makeatother

\newcommand{\ca}[1]{\mathcal{#1}}
\newcommand{\abs}[1]{\left| #1 \right|}
\newcommand{\norm}[1]{\left\| #1 \right\|}
\newcommand{\var}{\textup{var}}

\newcommand{\ep}{\varepsilon}

\newcommand{\N}{\mathbb{N}}   
\newcommand{\R}{\mathbb{R}}    
\newcommand{\goth}[1]{\mathfrak{#1}} 

\newcommand{\quadraVari}[1]{\left\langle  #1  \right\rangle } 

\newcommand{\fle}{\rightarrow}

\newcommand\restr[2]{{
  \left.\kern-\nulldelimiterspace 
  #1 
  \vphantom{\big|} 
  \right|_{#2} 
  }}

\title{Riesz spaces of signed charges on semi-rings}
\author{S. Cambronero\,\orcidlink{0000-0001-6758-4942}$^1$}
\author{D. Campos\,\orcidlink{0000-0002-3608-1151}$^2$}
\author{C. A. Fonseca-Mora\,\orcidlink{0000-0002-9280-8212}$^3$}
\author{D. Mena\,\orcidlink{0000-0002-9443-391X}$^4$}
\address{Centro de Investigaci\'{o}n en Matem\'{a}tica Pura y Aplicada \\ Escuela de Matem\'{a}tica, Universidad de Costa Rica}
\email{$^1$ santiago.cambronero@ucr.ac.cr} 
\email{$^2$ josedavid.campos@ucr.ac.cr}
\email{$^3$ christianandres.fonseca@ucr.ac.cr}
\email{$^4$ dario.menaarias@ucr.ac.cr}

\begin{document}

\emergencystretch 3em


\begin{abstract}
A constructive definition of the supremum of a family of set functions is exploited in the context of Riesz spaces of signed measures and finitely additive functions (signed charges) on semi-rings.  We explore applications, particularly to establish a Jordan decomposition for signed charges on semi-rings, whether the structure of Riesz space is present or not. 
\end{abstract}


\maketitle




\section{Introduction}

Roughly speaking, a \emph{vector lattice} (or a \emph{Riesz space}) is a vector space which is also a lattice, in such a way that both structures are compatible. If moreover, the vector lattice has a norm which is compatible with the order structure (a lattice norm), and this normed space is complete, we call it a \emph{Banach lattice}. Many classical spaces of importance in modern functional analysis are vector and Banach lattices. In particular, it is well known since \cite{Alexandroff:1940} and \cite{Alexandroff:1943} that the space of bounded variation signed charges over an algebra is a Banach lattice.

Vector lattices of signed measures and charges have been around, at least, since the $1930$'s (see, for instance, \cite{Bochner:1941Phillips, Fichtenholz:1934Kantorovitch, Hildebrandt:1934, Riesz:1940}). Indeed, the starting point of the theory of vector lattices is usually attributed to the contribution of Riesz at the ICM in 1928 (see \cite{Riesz:1928}). Banach lattices were first considered by Kantorovich in \cite{Kantorovitch:1937}. Since these early developments,  the area has expanded and has developed connections with several areas of mathematics, as well as applications to other branches of science, for example, to mathematical economics and mathematical finance. \medskip

The purpose of this paper is to study some properties of Riesz spaces of signed measures and finitely additive functions (signed charges) on semi-rings. We do this by considering a constructive definition of the supremum of a family of set functions.

Despite the large literature on vector and Banach lattices (e.g. see  \cite{Alip:2006, Alip:1998, Dunford:1988Schwartz, Meyer:1991, Rao1983Rao}  among others),  few attempts have been made to generalize results to the case of signed measures and charges on semi-rings. One example is \cite{Ercan:2009}, which uses a non-constructive way for the definition of the supremum. In that work, signed measures on $\sigma$-semi-rings, and signed charges on semi-rings are considered, but these are defined as differences of measures and charges from the beginning.\medskip


The construction of a lattice comes down to showing that, given two elements $\mu,\nu$ in the space, there exists a third element denoted by $\mu\lor\nu$, which is the least upper bound of the pair $\{\mu,\nu\}$. For spaces of signed measures this construction has been known for almost a century. However, to prove order completeness of the space, explicit constructions have not been so popular. Interestingly enough, the idea of an explicit construction was already presented in the early forties by Riesz (see \cite{Riesz:1940}) as mentioned by \cite{Pietsch:2007Hist}. In fact, Riesz proves that the dual of Banach lattice is a  Dedekind complete Riesz space. For this, he explicitly constructs the supremum of a family of functionals over the Banach lattice. We find an explicit construction of the infimum of a family of charges over a field $\Sigma$ in \cite[Theorem 5, p. 162]{Dunford:1988Schwartz}, and the corresponding proof of the order completeness of the lattices $\textup{ba}(\Sigma)$ and $\textup{ca}(\Sigma)$ in their Corollary 6.

Vector lattice spaces of signed measures and charges are typically built over an algebra of sets. This has the advantage that every finite measure becomes of bounded variation. This advantage, however, restricts the number of examples and applications. Signed measures on semi-rings, on the other hand, can be finite or even of finite variation, without being bounded. This fact, enriches the number of situations in which they can be applied.

In Sections \ref{sectionCharges} and \ref{sectionMeasures} we present an explicit construction of the supremum of a family of set functions defined on a semi-ring, with a particular emphasis on families of signed charges (Section \ref{sectionCharges}) and of signed measures (Section \ref{sectionMeasures}) defined on a ring.  Some examples are given so the reader can get an early idea of the usefulness of these results. 

In Section \ref{sectionLattices} we explicitly construct new spaces of signed measures and charges over semi-rings, and apply such a construction to extend some classical theorems to this context.

Sections \ref{sectJordanFiniteVariation} and \ref{sectJordanGeneralCase} are dedicated to the study of the Jordan decomposition for a signed charge on a semi-ring. This is done first in Section \ref{sectJordanFiniteVariation} for finite charges of finite variation. The Jordan decomposition is obtained as a direct consequence of our results in the previous sections and the general theory of Dedekind complete Riesz spaces. In Section \ref{sectJordanGeneralCase} the Jordan decomposition is extended to signed charges which might have no finite variation, or that might even take infinite values. As the reader will see, the classical proofs are simplified by using our explicit construction of the supremum. In Section \ref{sectHahnDecomposition} we prove an $\epsilon$-Hahn decomposition for signed charges defined on a semi-ring. Our result extends the one obtained in  \cite{Rao1983Rao}.  

Finally, in Section \ref{sectApplicationsDensity} we apply our results in the previous sections to provide and explicit formula for the supremum of a family of signed measures defined by integration with respect to some extended real-valued measurable functions (see \ref{theoIntegralOfSupremum}).  This generalizes a previous result in \cite{VeraarYaroslavtsev:2016}, for measures defined as integrals of non-negative measurable functions. 

\section{The supremum of a family of signed charges}
\label{sectionCharges}

Consider a set $\Omega$ and a semi-ring $\ca{S}$ of subsets of $\Omega$. This means that $\emptyset\in\ca{S}$, $\ca{S}$ is closed under finite intersections and, given $A,B\in \ca{S}$, $A \setminus B$ is a finite disjoint union of elements of $\ca{S}$. A set function $\mu:\ca{S}\fle \overline{\R}$ is called a charge if it is (finitely) additive and avoids one of the values $+\infty$ or $-\infty$; when it is also countably additive we call it a measure. We call $\mu$ sub-additive if it satisfies
$$
\mu\left( \cup A_i \right) \leq \sum \mu(A_i)
$$
for finite families of pairwise disjoint elements of $\ca{S}$ whose union also belongs to $\ca{S}$. When this is valid for countable families we say that $\mu$ is countably sub-additive. Likewise, we define super-additive and countably super-additive measures.

Consider a family $F = (\nu_j)_{j\in J}$ of set functions defined on $\ca{S}$. We assume that $F$ satisfies the following \emph{admissibility hypothesis}:
\begin{equation}
\label{eqAdmisible}
\text{For all } A\in \ca{S}\quad \sup_{j\in J} \nu_j(A) > -\infty.
\end{equation}
For instance, if there is $j\in J$ such that $\nu_j > -\infty$ (and even more particularly, if all these set functions are finite) the family $F$ satisfies (\ref{eqAdmisible}).

\begin{example}
For every $n \in \N$ let $\alpha_n : \ca{B}(\R) \rightarrow \overline{\R}$ be defined by 
$$
\alpha_{n}(A) := Leb(A\cap [0,\infty)) - Leb(A\cap [-n,0)).
$$    
The family $(\alpha_n)_{n\geq 1}$  satisfies (\ref{eqAdmisible}).
\end{example}

\begin{definition}
Consider a family $F = (\nu_j)_{j\in J}$ of set functions defined on a semi-ring $\ca{S}$. If $F$ satisfies (\ref{eqAdmisible}), we can define a set function $\mu_F$ on $\ca{S}$ by
$$
\mu_F(A)  := \sup_{ \Pi\in \mathscr{P}(A) } \sum_{C\in \Pi}\sup_{j \in J}\nu_j(C), \qquad A \in \ca{S}
$$
where $\mathscr{P}(A)$ is the family of all finite partitions of $A$ by elements of $\ca{S}$.
\end{definition}

\begin{remark}
    Because of condition (\ref{eqAdmisible}), $\mu_F$ is well defined and $\mu_F > -\infty$.
\end{remark}

We consider the natural order relation for set functions defined on $\ca{S}$:
$$
\mu \leq \nu \Leftrightarrow \mu(A) \leq \nu(A) \text{ for all }A\in\ca{S}.
$$
This is clearly a partial order. We are especially interested in this order relation, when restricted to the set $c(\ca{S})$ of all signed charges defined on $\ca{S}$. As usual, the least upper bound of a bounded $F\subseteq c(\ca{S})$, when it exists, is denoted by $\sup F$.

\begin{theorem}
    \label{theoSupCharges}
    Let $F$ be a family of set functions defined on a semi-ring $\ca{S}$. If $F$ satisfies (\ref{eqAdmisible}), then $\mu = \mu_F$ is super-additive. Moreover, if $F \subseteq c(\ca{S})$ then $\mu_F  \in c(\ca{S})$ and $\mu_F = \sup F$.
\end{theorem}
\begin{proof}
Let $A_1,\ldots A_n \in \ca{S}$ be disjoint and such that $A_1\cup \ldots \cup A_n = A \in \ca{S}$. We know that $\mu(A_k) \in (-\infty,\infty]$, so let $c \in \R$ such that $c < \mu(A_1) +\ldots + \mu(A_n)$. Choose $a_k < \mu(A_k)$ for $k=1,\ldots,n$ such that $a_1+\ldots+a_n = c$. By definition, there are partitions $\Pi_k \in \mathscr{P}(A_k)$ such that
$$
a_k < \sum_{C\in \Pi_k} \sup_{j\in J} \nu_j(C),\quad k=1,\ldots,n.
$$
Since $\Pi : = \Pi_1\cup \ldots \cup \Pi_n \in \mathscr{P}(A)$, it follows that
$$
c=\sum_{k=1}^n a_k < \sum_{k=1}^n \sum_{C\in \Pi_k}  \sup_{j\in J} \nu_j(C) = \sum_{C\in \Pi}  \sup_{j\in J} \nu_j(C)\leq \mu(A).
$$
Since $c$ is arbitrary we obtain $\mu(A_1) +\ldots+ \mu(A_n) \leq \mu(A)$. Therefore, $\mu$ is super-additive. 

We now assume that $F\subseteq c(\ca{S})$. For the sub-additivity of $\mu$, consider $A_1,\ldots A_n \in \ca{S}$ be disjoint and such that $A_1\cup \ldots \cup A_n = A \in \ca{S}$.  If $\Pi\in \mathscr{P}(A)$, it follows that $C\cap A_n \in \ca{S}$ for each $C\in \Pi$. By the additivity of each $\nu_j$, and the super-additivity of $\mu_F$ one gets
$$
\sum_{C\in \Pi} \sup_{j \in J} \nu_j(C) = 
\sum_{C\in \Pi} \sup_{j \in J} \sum_{k=1}^n \nu_j(C\cap A_n) \leq 
\sum_{k=1}^n \sum_{C\in \Pi} \mu(C\cap A_n) \leq 
\sum_{k=1}^n  \mu(A_n).
$$ 
Taking the supremum over all partitions of $A$ on the left hand side, we obtain the sub-additivity for $\mu$.  Finally, if $\lambda \in c(\ca{S})$ is such that $\nu_j\leq\lambda$ for each $j\in J$ we have
$$
\mu(A) = \sup_{\Pi\in \mathscr{P}(A)}\sum_{C\in\Pi} \sup_{j\in J} \nu_j(C) \leq \sup_{\Pi\in \mathscr{P}(A)}\sum_{C\in\Pi} \lambda(C) = \lambda(A)
$$
for any $A\in \ca{S}$ and therefore $\mu \leq \lambda$.
\end{proof}

The above results can be applied to quite general situations. For instance, $\ca{S}$ could be so small that no partition other that $\Pi = \{ A \}$ exists for each set $A\in \ca{S}$.

\begin{example}
    Consider $\ca{S} = \{ \emptyset, A_1,\ldots,A_n \}$, 
    where $n>1$ and $A_1,\ldots,A_n$ form a partition of $\Omega$. Any set function $\nu$ defined on $\ca{S}$ is countably additive. In this case $\mu_F(A) = \sup_{j\in J} \nu_j(A)$.
\end{example}

The next corollary follows immediately from Theorem \ref{theoSupCharges}.

\begin{corollary}
    \label{coroSupOfFiniteSignedCharges}
    Let $F$ be a family of finite signed charges on a semi-ring $\ca{S}$, and suppose there is a finite charge $\mu$ on $\ca{S}$ that dominates each $\nu_j$.  Then $\sup F$ is a finite signed charge.
\end{corollary}

\section{The supremum of a family of signed measures}
\label{sectionMeasures}
In this section, we will need the following lemma.

\begin{lemma}
\label{lemmaMeasureSemiRing}
    Consider a charge $\mu$ defined on a semi-ring $\ca{S}$. If $\mu$ is countably sub-additive, then it is a measure.
\end{lemma}
\begin{proof}
 Let $(A_n)$ be a sequence of pairwise disjoint sets in $\ca{S}$ such that 
 $$
 A := \bigcup_{n=1}^\infty A_n \in \ca{S}.
 $$
By an induction argument, for each $n$ there are disjoint sets $D_1,\ldots D_m \in \ca{S}$ such that
 $$
 A \setminus \bigcup_{j=1}^n A_j = \bigcup_{k=1}^m D_k.
 $$
Since $\mu$ is a charge this gives
$$
\mu(A) = \sum_{j=1}^n \mu(A_j) + \sum_{k=1}^m \mu(D_j) \geq \sum_{j=1}^n \mu(A_j).
$$
When $n\uparrow\infty$ we obtain countable super-additivity and we are done.
\end{proof}

\begin{lemma}
\label{lemmaSubaddSupremum}
Consider a family $F = (\nu_j)_{j\in J}$ of set functions defined on a semi-ring $\ca{S}$, that satisfies (\ref{eqAdmisible}). If each $\nu_j$ is countably sub-additive, then so is $\mu_F$.
\end{lemma}
\begin{proof}
Let $(A_n)$ be a sequence of pairwise disjoint sets on $\ca{S}$ such that $A = \bigcup A_n \in \ca{S}$. We can use the same argument as in the proof of Theorem \ref{theoSupCharges}, with this countable collection and considering the countable sub-additivity of each $\nu_j$.
\end{proof}

\begin{theorem}
\label{theoSubaddSupremum}
Suppose $F\subseteq c(\ca{S})$ satisfies (\ref{eqAdmisible}), each $\nu_j$ is countably sub-additive, and there is a finite measure $\lambda$ such that $\mu_F \geq -\lambda$. Then $\mu_F$ is a signed measure.
\end{theorem}

\begin{proof}
 From Theorem \ref{theoSupCharges} and Lemma \ref{lemmaSubaddSupremum}, it follows that $\mu_F + \lambda$ is a charge and it is countably sub-additive. Therefore (from Lemma \ref{lemmaMeasureSemiRing}), $\mu_F+\lambda$ is a measure, so $\mu_F$ is a signed measure. 
\end{proof}

\begin{corollary}
    \label{coroSupSignedMeasures}
    Let $(\nu_j)$ be a family of finite signed measures on a semiring $\ca{S}$. If there is a finite signed measure that dominates each $\nu_j$, then $\mu:= \sup\nu_j$ is a finite signed measure.
\end{corollary}

We finish this section by observing that, if
$$
\inf_{j\in J} \nu_j(C) < \infty, \quad \text{for all }C\in \ca{S} 
$$
we can define
$$
\inf_{j\in J} \nu_j := - \sup _{j\in J} (-\nu_j).
$$
We immediately get
$$
\left( \inf_{j\in J} \nu_j \right)(A) = \inf_{\Pi\in \mathscr{P}(A)} \sum_{C\in \Pi} \inf_{j\in J} \nu_j(C).
$$

\section{Riesz spaces of charges and measures on semi-rings}
\label{sectionLattices}

Let $\Omega$ be a set and $\ca{A}$ an algebra of subsets of $\Omega$. In  \cite{Alip:2006}, they consider the space $ ba(\ca{A})$ whose elements are finite signed charges. This is a real vector space with the natural operations
$$
(\mu+\nu)(A) = \mu(A)+\nu(A),\quad (\alpha \mu)(A) = \alpha \mu(A).
$$
Besides, the natural order completes the structure of Riesz space. All of this is proven in \cite{Alip:1998} and \cite{Alip:2006}, and furthermore, it is shown that $ba(\ca{A})$ is a Banach lattice with the norm given by $\norm{\mu} = |\mu|(\Omega)$. \medskip

In this section, we generalize these results to the case in which the family of sets is only a semi-ring $\ca{S}$, and the signed charges considered are only assumed of finite variation. More precisely, a charge $\mu$ on $\ca{S}$ is of finite variation if the charge
$$
|\mu| := \sup \{-\mu,\mu \}
$$
is finite (takes only finite values). This does not prevent it from taking arbitrarily large values.

\begin{example}[A bounded signed measure, with finite unbounded variation] \label{examChargeOnSemiringOfNaturals}

On $\Omega=\N$ consider
$$
\mathcal{S}=\{ A_{m,n}: m,n\in \N \},\quad A_{m,n} = \{ k\in \N: m \leq k < n \}.
$$
This is clearly a semi-ring. Define $\mu:  \mathcal{S} \rightarrow \R$ by 
$$
\mu(A)= \sum_{k \in A} \frac{(-1)^k}{k}.
$$
Notice that $\mu(\emptyset)=0$ and for $m<n$, $\left| \mu(A_{m,n}) \right| \leq \frac{1}{m} \leq 1.$
It is clear that $\mu$ is a bounded signed measure and its variation $|\mu|$ is given by
$$
|\mu|(A) = \sum_{n\in A}\frac{1}{n}.
$$
We conclude that $\mu$ is a bounded signed measure with finite unbounded variation.
\end{example}
The following example gives a signed charge that is not of finite variation.

\begin{example}[A finite signed charge with no finite variation] \label{examChargeOnAlgebraOfNaturals}
This example is based on Example 10.7 in \cite{Alip:2006}. On $\Omega=\N$ consider the algebra 
$$
\mathcal{A}=\{ A \subseteq \N: A \mbox{ or } A^{c} \mbox{ is finite } \}.
$$
Define $\mu:  \mathcal{A} \rightarrow \R$ by 
$$\mu(A)=
\begin{cases}
\text{card}(A) & \mbox{if }A\text{ is finite} \\
-\text{card}(A^{c}) & \mbox{otherwise}. 
\end{cases}
$$
It is not difficult to check that $\mu$ is a signed charge on $\mathcal{A}$, but not a signed measure. For example, 
 $$
\mu\left( \bigcup_{n=3}^\infty \{ n\} \right) = -2 \neq \sum_{n=3}^\infty \mu(\{ n\}).$$
Even though $\mu$ has unbounded range, it is a finite signed charge. However, 
 $$
 |\mu|(\N) \geq |\mu\left( \{ 1,2, \dots, n\} \right)| + |\mu\left( \{n+1, n+2, \dots \} \right)| = 2n
 $$
for every $n\in\N$, so $|\mu|(\N)=\infty$. Even more, 
$|\mu|(A) = \infty$ whenever $A^c$ is finite. Hence, we have shown that $\mu$ is a finite charge but not of finite variation.
\end{example}

\subsection{The space of finite variation signed charges over a semi-ring}\ 

We denote by $\textup{fa}(\ca{S})$ the space of finite variation signed charges defined on $\ca{S}$. More precisely
$$
\textup{fa}(\ca{S}) := \{ \mu:\ca{S}\fle \R : \mu \text{ is a signed charge and } |\mu| \text{ is finite} \}.
$$
For instance, the signed charge of Example \ref{examChargeOnSemiringOfNaturals} belongs to $\textup{fa}(\ca{S})$, but the one of Example \ref{examChargeOnAlgebraOfNaturals} does not. Notice that, in general, if $\mu\in \textup{fa}(\ca{S})$ then $|\mu| \in \textup{fa}(\ca{S})$, because $|\mu|$ is its own variation. 
The space $\textup{fa}(\ca{S})$ is a Riesz space with the usual operations and usual order. In fact, given $\mu,\nu\in \textup{fa}(\ca{S})$ we have $\mu,\nu \leq |\mu|+|\nu|$ and by Corollary \ref{coroSupOfFiniteSignedCharges}, $\mu \lor \nu := \sup \{\mu,\nu\}$ is a finite signed charge satisfying
$$
\mu \leq \mu \lor \nu \leq |\mu|+|\nu|. 
$$
Thanks to Corollary \ref{coroSupOfFiniteSignedCharges}, every nonempty upper bounded subset of $\textup{fa}(\ca{S})$ has a supremum, which belongs to $\textup{fa}(\ca{S})$. This shows the following.

\begin{theorem}
\label{teoRieszSpaceCharges}
  The space $\textup{fa}(\ca{S})$ is a Dedekind complete Riesz space.
\end{theorem}

\subsection{The space of finite variation signed measures over a semi-ring}\ 

We denote $\textup{ca}(\ca{S})$ the space of all finite variation signed measures defined on the semi-ring $\ca{S}$. More precisely
$$
\textup{ca}(\ca{S}) := \{ \mu:\ca{S}\fle \R : \mu \text{ is a signed measure and } |\mu| \text{ is finite} \}.
$$
Similarly to the case of $\textup{fa}(\ca{S})$, whenever $\mu\in \textup{ca}(\ca{S})$ we have $|\mu| \in \textup{ca}(\ca{S})$. Besides, the space $\textup{ca}(\ca{S})$ is a Riesz space with the usual operations and usual order. Thanks to Corollary \ref{coroSupSignedMeasures}, every nonempty upper bounded subset of $\textup{ca}(\ca{S})$ has a supremum, which belongs to $\textup{ca}(\ca{S})$. This shows the following.

\begin{theorem}
\label{teoremaRieszSpaceMeasures}
  The space $\textup{ca}(\ca{S})$ is a Dedekind complete Riesz space.
\end{theorem}

\subsection{Bounded variation signed charges and measures}\ 

Consider the space 
$$
\textup{ba}(\ca{S}) := \{ \mu\in \textup{fa}(\ca{S}): |\mu| \text{ is bounded} \}
$$
and define 
$$
\norm{\mu} := \sup_{A\in \ca{S}} |\mu|(A).
$$
This is clearly a norm and $\textup{ba}(\ca{S})$ is norm complete. The norm-completeness follows as in the proof of Theorem $10.53$ of \cite{Alip:2006}. 

\begin{remark}
    If $\Omega \in \ca{S}$, according to Definition $9.26$ in \cite{Alip:2006}, $\textup{ba}(\ca{S})$ is an AL-space. In fact, for $\mu \geq 0$ and $\nu\geq 0$  we have $
\norm{\mu+\nu} = \mu(\Omega)+\nu(\Omega) = \norm{\mu} + \norm{\nu}.$  It follows from the theorem of Kakutani (Theorem $9.33$ in \cite{Alip:2006}) that $\textup{ba}(\ca{S})$ is lattice isometric to some $L_1(\mu)$-space. 
\end{remark}

\subsection{Projection band}\

Given a Riesz space $E$, a vector subspace of $E$ is called Riesz subspace if it is closed under lattice operations. Equivalently, a vector subspace $F$ is a Riesz subspace if $|\mu|\in F$ whenever $\mu\in F$. It follows directly from the definition, that $\textup{ca}(\ca{S})$ is a Riesz subspace of $\textup{fa}(\ca{S})$. If a vector subspace $F$ is also solid ($|x|\leq |y|$ and $y\in F$ implies $x\in F$), it is called a an ideal. It is clear that every ideal on $E$ is a Riesz subspace of $E$. Thanks to Theorem 8.13 in \cite{Alip:2006}, a Riesz subspace $F$ is an ideal if and only if
$$
0\leq x\leq y, y\in F \Rightarrow x\in F.
$$
Finally, an ordered closed ideal is called a band. We now extend Theorem 10.56 of \cite{Alip:2006}.

\begin{theorem}
For a semi-ring $\mathcal{S}$, $\textup{ca}(\ca{S})$ is a projection band on $\textup{fa}(\ca{S})$. In particular, we have the decomposition 
$$
\textup{fa}(\mathcal{S}) = \textup{ca}(\mathcal{S})\oplus  \textup{ca}(\mathcal{S})^d
$$
\end{theorem}

\begin{proof}
By Theorem 8.20 in \cite{Alip:2006}, it is enough to show that $\textup{ca}(\mathcal{S})$ is a band. To prove that $\textup{ca}(\mathcal{S})$ is an ideal, by Theorem 8.13 in \cite{Alip:2006}, it suffices to show that for $\nu \in \textup{fa}(\mathcal{S})$ such that $0 \leq \nu \leq \mu$, with $\mu \in \textup{ca}(\mathcal{S})$, then $\nu \in \textup{ca}(\mathcal{S})$.   Let $(A_n)_{n \in \mathbb{N}}$ a sequence of disjoints sets in $\mathcal{S}$ such that $A = \bigcup_{n = 1}^{\infty} A_n \in \mathcal{S}$.   Since we can write $A = \left( \bigcup_{k = 1}^n A_k\right) \cup \left(\bigcup_{j=1}^m D_j \right)$, with $D_1, \ldots, D_m \in \mathcal{S}$ disjoint, then
$$
0 \leq \nu(A) - \sum_{k=1}^n \nu(A_k) = \sum_{j=1}^m \nu(D_j) \leq  \sum_{j=1}^m \mu(D_j) = \mu(A) - \sum_{k=1}^n \mu(A_k),
$$
and the right hand side goes to zero as $n$ goes to infinity.  This shows that $\nu \in \textup{ca}(\mathcal{S})$. 

To conclude that $\textup{ca}(\mathcal{S})$ is a band, following the arguments in section 8.9 in \cite{Alip:2006}, it is enough to show that if $\{ \nu_{\alpha} \} \subseteq \textup{ca}(\mathcal{S})$ and $0 \leq \nu_{\alpha} \uparrow \nu$, then $\nu \in \textup{ca}(\mathcal{S})$, but this is an immediate consequence of Corollary \ref{coroSupSignedMeasures} since $\nu = \sup_{\alpha} \nu_{\alpha} \in \textup{ca}(\mathcal{S})$.
\end{proof}

\section{Jordan decomposition as a consequence of the structure of Riesz space}\label{sectJordanFiniteVariation}

We exploit results of last section and the general theory of Dedekind complete Riesz spaces, to simplify proofs and extend some classical theorems. We start with a Jordan decomposition theorem for signed charges and signed measures of finite variation on semi-rings.

\begin{theorem}
\label{teoJordanSemirings}
    Let $\mu$ be a signed charge of finite variation on a semiring $\ca{S}$. There are unique finite charges $\mu^+$ and $\mu^-$ such that 
    $$
    \mu = \mu^+ - \mu^-,\quad \mu^+\land \mu^- = 0.
    $$
    They are given by $\mu^+ := \mu \lor 0$ and $\mu^- := (\mu)\lor 0$. We also have $|\mu| = \mu^+ + \mu^-$. If $\mu$ is a signed measure, then $\mu^+$ and $\mu^-$ are measures.
\end{theorem}
This follows immediately from Theorems \ref{teoRieszSpaceCharges}, \ref{teoremaRieszSpaceMeasures} and Theorems $8.6$ and $8.11$ in \cite{Alip:2006}.

\begin{example}[A bounded signed measure with finite unbounded variation on a semi-ring]

Let $\Omega=[0,1)$ and consider the semi-ring $\mathcal{A}=\{ [a,b): 0 < a \leq b \leq 1\}$. Define $\mu:  \mathcal{A} \rightarrow \R$ by 
$$
\mu([a,b))=\sum_{n=1}^{\infty} (-1)^{n+1}(n+1) \mbox{Leb} \left( [a,b) \cap \left[ \frac{1}{n+1}, \frac{1}{n} \right) \right). 
$$
Since $a>0$, this sum always has a finite number of terms. One can easily check that $-\frac{1}{2} \leq \mu[a,b) \leq 1$. Suppose
$$
[a,b) = \bigcup_{n=1}^\infty [a_n,b_n) \in \ca{S}.
$$
Since $a>0$ and $\mu$ is a difference of two finite measures on $[a,b)$, we obtain 
$$
\mu[a,b) = \sum_{n=1}^\infty \mu[a_n,b_n).
$$
Observe that $\mu$ is of finite (though unbounded) variation since 
$$ 
\mu^{+} \left( \left[ \frac{1}{2n}, 1 \right) \right) = \sum_{k=1}^n \frac{1}{2k-1}, \quad
\mu^{-} \left( \left[ \frac{1}{2n}, 1 \right) \right) = \sum_{k=1}^n \frac{1}{2k}.
$$
We conclude that $\mu \in \textup{ca}(\ca{S})$ and $|\mu|$ is an unbounded finite measure on $\ca{S}$.
\end{example}

\section{Jordan decomposition for general signed charges on semi-rings}
\label{sectJordanGeneralCase}

In last section, we took advantage of properties of the supremum of set functions and the general theory of complete Riesz spaces. The structure of Riesz space requires the variation of its elements to be finite, because the variation itself is a supremum of two elements of the space. But we can extend results from last section to more general contexts, where the lattice structure is not present. Results on this section can be seen as extending Proposition $2.5.3$ of \cite{Rao1983Rao} to the case of a semi-ring as underlying family.\medskip

We start with a version of the Jordan decomposition theorem for finite signed measures and charges on semi-rings, regardless of the variation.

\begin{theorem}
\label{teoJordanFiniteSigned}
    Let $\mu$ be a finite signed charge on a semiring $\ca{S}$. The charges $\mu^+ := \mu\lor 0$ and $\mu^- := (-\mu)\lor 0$ satisfy
    \begin{equation}
      \label{eqJordanImplicito}
      \mu^+ = \mu + \mu^-,\quad |\mu| = \mu^+ + \mu^-.
    \end{equation}
    If $\mu$ is a finite signed measure, then $\mu^+$ and $\mu^-$ are measures. If one of the charges $\mu^+$ or $\mu^-$ is finite, then so is the other and Theorem \ref{teoJordanSemirings} applies.
\end{theorem}

\begin{proof}
    The classical lattice proof applies, except for the fact that we cannot subtract $\mu^-$ from $\mu^+$. In fact, since $\mu(A) \in \R$ for each $A$ we have 
    $$
    \mu^+ = \sup \{ \mu, 0 \} = \sup \{ \mu - \mu, -\mu \} + \mu = \mu^- + \mu
    $$
    and 
    $$
    |\mu| = \sup\{ 2\mu,0\}-\mu = 2\mu^+-\mu = \mu^+ + \mu^-.
    $$
    The last assertion follows immediately from the first equality in (\ref{eqJordanImplicito}).
\end{proof}

\begin{example}
   Consider the finite signed charge of example \ref{examChargeOnAlgebraOfNaturals}. For $A$ finite we have
   $$
   \mu^+(A) = \text{card}(A) = \mu(A),\quad \mu^-(A)=0,
   $$
   while $\mu^+(A) = \mu^-(A) = \infty$ for $A^c$ finite. The first identity in (\ref{eqJordanImplicito}) becomes $\infty = -\text{card}(A^c) + \infty$.
\end{example}

We now consider the most general case, in which $\mu$ can take an infinite value. 
\begin{theorem}
\label{teoJordaSignedInfinite}
     Let $\mu$ be a signed charge on a semi-ring $\ca{S}$, with $\mu > -\infty$. The charges $\mu^+ := \mu\lor 0$ and $\mu^- := (-\mu)\lor 0$ satisfy
    \begin{equation}
      \mu^+ = \mu + \mu^-,\quad |\mu| = \mu^+ + \mu^-.
    \end{equation}
    If $\mu$ is a signed measure, then $\mu^+$ and $\mu^-$ are measures.
\end{theorem}

\begin{proof}
    If $\mu(A)$ is finite, it is clear that $\mu(C)$ is finite for any $C\subseteq A$. Now the family
    $$
    \ca{S}_A := \{ A\cap B:B\in \ca{S} \}
    $$
    is a semiring and $\mu_A := \left. \mu \right|_{\ca{S}_A}$ is a finite signed charge on $\ca{S}_A$. By Theorem \ref{teoJordanFiniteSigned} it follows that
    $$
    \mu^+(A) = \mu(A) + \mu^-(A), \quad |\mu|(A) = \mu^+(A) + \mu^-(A).
    $$
    If $\mu(A) = \infty$, there is nothing to prove.
\end{proof}

\begin{corollary}
     Let $\mu$ be a signed charge on a semiring $\ca{S}$, with $\mu < \infty$. The charges $\mu^+ := \mu\lor 0$ and $\mu^- := (-\mu)\lor 0$ satisfy
    \begin{equation}
    \label{eqJordanNonInfinite}
      \mu^+ - \mu = \mu^-,\quad |\mu| = \mu^+ + \mu^-.
    \end{equation}
    If $\mu$ is a signed measure, then $\mu^+$ and $\mu^-$ are measures.
\end{corollary}

\begin{proof}
   It is enough to apply Theorem \ref{teoJordaSignedInfinite} to $-\mu$.
\end{proof}

\begin{example}
   Let's modify example \ref{examChargeOnAlgebraOfNaturals}. Define
$$
\begin{cases}
\text{card}(A) & \mbox{if }A\text{ is finite} \\
- \infty & \mbox{otherwise}. 
\end{cases}
$$
In this case, $\mu$ is a signed charge with $\mu < \infty$. For $A$ finite we still have $\mu^+(A) = \text{card}(A) = \mu(A)$ and $\mu^-(A)=0$. For $A^c$ finite we now have
$$
\mu^+(A) = \infty = \mu^-(A) = \infty.
$$
The first identity in (\ref{eqJordanNonInfinite}) becomes $\infty - (-\infty) = \infty$ is this last case.
\end{example}

The following lemma generalizes the result of Theorem $2.5.3(7)$ in \cite{Rao1983Rao}.

\begin{lemma}
     Let $\mu$ be a signed charge on a semi-ring $\ca{S}$, and let $A\in \ca{S}$. Then $(\mu^+\land \mu^-)(A)= 0$ if and only if, either $\mu^+(A)$ or $\mu^-(A)$ is finite.
\end{lemma}

\begin{proof}
    Suppose $\mu^-(A)<\infty$. As in the proof of Theorem \ref{teoJordaSignedInfinite}, on $(A,\ca{S}_A)$ we have
    $$
    \mu^+\land \mu^- = (\mu^+-\mu^-)\land 0 + \mu^- = \mu \land 0 + \mu^- = 0.
    $$
    In case $\mu^+(A)<\infty$, apply this to $-\mu$. We have proven the ``if'' part. \medskip

    For the ``only if'' let's assume $\mu^+(A) = \mu^-(A)=\infty$. In case $\mu>-\infty$ we have, for any $\Pi\in \mathscr{P}(A)$
    \begin{eqnarray*}
     \sum_{C\in\Pi} \min\{ \mu^+(C),\mu^-(C)\} &=& \sum_{C\in\Pi} \min\{ \mu(C) + \mu^-(C),\mu^-(C)\} \\
     &=& \sum_{C\in\Pi} \left( \min\{ \mu(C),0\} + \mu^-(C) \right) \\
    &=& \sum_{C\in\Pi} \min\{ \mu(C),0\} + \mu^-(A) = \infty.
    \end{eqnarray*}
   We have used Theorem \ref{teoJordaSignedInfinite}. Therefore $(\mu^+\land \mu^-)(A)=\infty$. When $\mu < \infty$, apply this to $-\mu$. 
\end{proof}

\begin{remark}
In the above argument it is actually proved that
    $$
     (\mu^+\land \mu^-)(A) = \left\{  \begin{array}{ll} 0 & \text{ if }\ \mu^+(A) < \infty \text{ or } \mu^-(A)<\infty \\
      \infty & \text{ if }\ \mu^+(A) = \mu^-(A) = \infty.
      \end{array} \right.
      $$
\end{remark}

\begin{example}
    In Example \ref{examChargeOnAlgebraOfNaturals}, for $A$ finite we have 
    $(\mu^+\land \mu^-)(A) = 0$, while $(\mu^+\land \mu^-)(A)=\infty$ when $A^c$ is finite.
\end{example}

\section{A Hahn decomposition for signed charges on semi-rings}\label{sectHahnDecomposition}

Consider the ring $\ca{R}$ generated by $\ca{S}$, that is, the family whose elements are finite unions of elements of $\ca{S}$.  Since $\ca{S}$ is a semi-ring, it is equivalent to consider finite disjoint unions. Given a signed charge $\mu$ defined on $\ca{S}$, denote by $\widehat{\mu}$ its unique extension to $\ca{R}$, which is naturally defined by
$$
\widehat{\mu}\left(\bigcup_{i=1}^n A_i \right) := \sum_{i=1}^n \mu(A_i),
$$
for pairwise disjoint $A_1,\ldots,A_n\in \ca{S}$. It is well known that, if $\mu$ is a signed measure, then so is $\widehat{\mu}$. 

\begin{lemma}
    Let $(\nu_j)_{j\in J}$ be a family of signed charges defined on $\ca{S}$ and satisfying the condition (\ref{eqAdmisible}). Let $\nu := \sup \nu_j$ (as a signed charge on $\ca{S}$) and $\mu:= \sup_{j\in J} \widehat{\nu}_j$ (as a signed charge on $\ca{R}$). Then $\mu = \widehat{\nu}$. In particular, for any signed charge on $\ca{S}$ we have $(\widehat{\mu})^+ = \widehat{\mu^+}$ and $(\widehat{\mu})^- = \widehat{\mu^-}$.
\end{lemma} 

\begin{proof}
    Let us fix $A\in \ca{S}$ and denote by $\mathscr{P}_{\ca{S}}$ the family of partitions of $A$ by elements of $\ca{S}$, and similarly for $\mathscr{P}_{\ca{R}}$. It is clear that $\mathscr{P}_{\ca{S}}\subseteq \mathscr{P}_{\ca{R}}$, so $\nu(A) \leq \mu(A)$. On the other hand, given $\Pi\in \mathscr{P}_{\ca{R}}$, each $C\in\Pi$ is a finite disjoint union of elements of $\ca{S}$, so there is a partition $\widehat{\Pi}\in \mathscr{P}_{\ca{S}}$ which is thinner than $\Pi$. Consequently
    $$
    \sum_{C\in \Pi} \sup_{j\in J} \widehat{\nu}_j(C) \leq \sum_{C\in \widehat{\Pi}} \sup_{j\in J} \widehat{\nu}_j(C) = \sum_{C\in \widehat{\Pi}} \sup_{j\in J} \nu_j(C)\leq \nu(A)
    $$
    and this implies $\mu(A) \leq \nu(A)$. We conclude that $\mu = \nu$ on $\ca{S}$, so $\mu = \widehat{\nu}$.
\end{proof}

Because of this lemma, we shall assume that signed charges (and in particular, signed measures) given on $\ca{S}$ are already defined on $\ca{R}$, so there is no need to use the notation $\widehat{\mu}$ any more.

We generalize definition $2.6.1$ of \cite{Rao1983Rao} to this context.

\begin{definition}
    Let $\ca{S}$ be a semi-ring of subsets of $\Omega$ and $\mu$ a signed charge defined on $\ca{S}$, and then on the ring $\ca{R}$ generated by $\ca{S}$. For $\ep>0$, a partition 
    $\{H,A\setminus H\}$ of $A\in \ca{R}$, where $H\in \ca{R}$, is called an $\ep$-Hahn decomposition for $\mu$ on $A$ if
    \begin{eqnarray*}
          \text{For } B\in\ca{R}, B\subseteq H \Rightarrow \mu(B) \leq \ep \\
    \text{For } D\in\ca{R}, D \subseteq A\setminus H \Rightarrow \mu(D) \geq -\ep.
    \end{eqnarray*}
\end{definition}

The following theorem generalizes Theorem on \cite{Rao1983Rao}, to the semi-ring case.

\begin{theorem}
    If either $\mu^+(A)<\infty$ or $\mu^-(A)<\infty$, then there exists an $\ep$-Hahn decomposition for $\mu$ on $A$. Otherwise ($\mu^+(A)=\mu^-(A)=\infty$), there exists $\ep>0$ such that no $\ep$-Hahn decomposition exists for $\mu$ on $A$.
\end{theorem}

\begin{proof}
    If $\mu^+(A)<\infty$ or $\mu^-(A)<\infty$, by last lemma we have $(\mu^+\land \mu^-)(A) = 0$. Given $\ep>0$, there is $\Pi\in \mathscr{P}(A)$ such that 
    $$
    \sum_{C\in\Pi} \min\{ \mu^+(C),\mu^-(C) \} < \ep.
    $$
    Define $\Pi_1 := \{ C\in\Pi: \mu^+(C) \leq \mu^-(C) \}$, $\Pi_2 := \Pi\setminus\Pi_1$ and
    $$
    H := \bigcup_{C\in\Pi_1}C\quad \text{so that}\quad A\setminus H = \bigcup_{C\in\Pi_2} C.
    $$
    If $B\in \ca{R}$ and $B\subseteq H$ we have
    $$
    \mu(B) \leq \mu^+(B) = \sum_{C\in \Pi_1} \mu^+(B\cap C) \leq \sum_{C\in \Pi_1} \mu^+(C) < \ep
    $$
    while for $D\in\ca{R}$ and $D\subseteq A\setminus H$ we have
     $$
    -\mu(D) \leq \mu^-(D) = \sum_{C\in \Pi_2} \mu^-(D\cap C) \leq \sum_{C\in\Pi_2} \mu^-(C) < \ep.
    $$
    On the other hand, the existence of an $\ep$-Hahn decomposition for $\mu$ on $A$, for every $\ep>0$, implies $(\mu^+\land \mu^-)(A)=0$. In fact, let $\ep>0$ and let $\{H,A\setminus H\}$ be the $\ep$-Hahn decomposition for $\mu$ on $A$. Since $\mu^+(H) \leq\ep$ and $\mu^-(A\setminus H) \leq \ep$ we get 
    $$
    (\mu^+\land\mu^-)(A) \leq 2\ep.
    $$
    By last lemma again, the last assertion follows.
\end{proof}



\section{Applications to signed measures with density}\label{sectApplicationsDensity}

We extend Lemma 2.8 on \cite{VeraarYaroslavtsev:2016} to the case of signed measures.

\begin{theorem}
\label{theoIntegralOfSupremum}
Let $(S,\ca{A},\nu)$ be a $\sigma$-finite measure space. Let $\goth{F}$ be a family of measurable functions from $S$ into $(-\infty,\infty]$, such that 
$$
\int f^{-} d\nu < \infty \textup{ for each } f\in \goth{F}.
$$
Let $(f_j)_{j\in\N}$ be a sequence in $\goth{F}$. Define $\check{f} = \sup_{j\geq 1} f_j $ and assume that $\sup_{f\in \goth{F}} f = \check{f}$. For each $f\in \goth{F}$ let $\mu_f$ be the signed measure defined by
$$
\mu_f(A) =\int_B f d\nu, \quad A\in \ca{A}.
$$
If we define $\check{\mu} := \sup_{f\in \goth{F}} \mu_f$, then $\check{\mu} = \sup_{j\geq 1} \mu_{f_j}$ and
\begin{equation}
\label{eqSupAndDensity}
\check{\mu}(A) = \int_A \check{f} d\nu.
\end{equation}
\end{theorem}
\begin{proof}
Clearly $\check{f}$ is measurable, so $A \mapsto \int_A \check{f} d\nu$ defines a signed measure that dominates each $\mu_f$. This gives the inequality ``$\leq$'' in \eqref{eqSupAndDensity}. 

For the other inequality, we adapt the proof of Lemma 2.8 in \cite{VeraarYaroslavtsev:2016}, which is established for non-negative functions and measures. Let $A\in \ca{A}$, $n\in\N$ and $0< \ep < 1$. Define
$$
A_1 =\{ s\in A: f_1(s) > (\check{f}(s)\land n)(1-\ep)>0 \}
$$
and
$$
A_{j+1} = \{ s\in A: f_{j+1}(s) > (\check{f}(s)\land n)(1 - \ep)>0 \}\setminus \bigcup_{i=1}^{j} A_i.
$$
The sets $A_j,\, j\geq 1$ are disjoint and cover $A\cap \{ \check{f}>0 \} $, so
$$
\check{\mu}(A\cap \{\check{f}>0\}) = \sum_{j\geq 1} \check{\mu}(A_j) \geq \sum_{j\geq 1} \mu_{f_j}(A_j) = \sum_{j\geq 1} \int_{A_j} f_j d\nu.
$$
It follows that
$$
\check{\mu}(A\cap \{\check{f}>0\}) \geq (1-\ep) \sum_{j\geq 1} \int_{A_j} (\check{f}\land n) d\nu
= (1-\ep)\int_{A\cap \{\check{f}>0 \}} (\check{f}\land n) d\nu.
$$
Since $\ep$ and $n$ are arbitrary, we obtain
$$
\check{\mu}(A\cap \{\check{f}>0\}) \geq \int_{A\cap \{\check{f}>0 \}} \check{f} d\nu.
$$
If we now use
$$
A_1 =\{ s\in A: \check{f}(s)<0, f_1(s)>\check{f}(s)(1+\ep)\}
$$
and
$$
A_{j+1} = \{ s\in A: \check{f}(s) < 0, f_{j+1}(s)>\check{f}(s)(1+\ep)\}\setminus \bigcup_{i=1}^{j} A_i
$$
we get
$$
\check{\mu}(A\cap \{\check{f}<0\}) \geq \int_{A\cap \{\check{f} < 0 \}} \check{f} d\nu.
$$
Now with $A_1 = \{ s\in A: -\ep < f_1(s)\leq 0 \}$ and 
$$
A_{j+1} = \{ s\in A: -\ep < f_{j+1}(s)\leq 0 \} \setminus \bigcup_{i=1}^j A_i
$$
we get 
$$
\check{\mu}(A\cap \{\check{f}=0\}) \geq -\ep \nu(A\cap \{\check{f} = 0\}).
$$
Since $\ep$ is arbitrary and (by localization) $\nu$ may be assumed finite, we get $\check{\mu}(A\cap \{\check{f}=0\}) \geq 0$.

Finally
$$
\check{\mu}(A) = \check{\mu}(A\cap \{\check{f} > 0\}) + \check{\mu}(A\cap \{\check{f}=0\}) + \check{\mu}(A\cap \{\check{f} < 0\}) \geq \int_A \check{f} d\nu.
$$
We have proven \eqref{eqSupAndDensity}. If we replace $\goth{F}$ with $(f_j)$, we obtain 
$$
\sup_{j\geq 1} \mu_{f_j} = \int_A \check{f} d\nu = \check{\mu}
$$  
and this completes the proof.
\end{proof}

\begin{remark}
    In the case the family in countable, the theorem comes down to
    $$
    \left( \sup_{n\in\N} \int_{\cdot} f_n d\nu \right)(A) = \int_{A} \left(\sup_{n\in\N} f_n \right) d\nu. 
    $$
\end{remark}

\begin{corollary}
\label{coroVarSignedMeasDensity}
Consider a signed measure $\mu$ given by
$$
\mu(A) := \int_A f(s)\, d\nu(s),
$$
where $\nu$ is a $\sigma$-finite measure. Then
$$
\mu^+(A) = \int_A f^+(s)\,d\nu(s),\quad
\mu^-(A) = \int_A f^-(s)\,d\nu(s),\quad
\var(\mu)(A) = \int_A \abs{f(s)}\,d\nu(s).
$$
\end{corollary}


\begin{example}
    Consider $f:\R\fle\R$ integrable and define $\alpha$ on $(\R,\ca{B}(\R))$ by
    $$ 
    \alpha(A) = \int_A f(t)dt.
    $$
    By Corollary \ref{coroVarSignedMeasDensity}, the variation of $\alpha$ is given by
    $$
    \var(\alpha)(A) = \int_A \abs{f(t)}dt.
    $$
    When applying this with $A = (a,x]$, we can use Riemann-type partitions in the definition of the supremum (see for instance Remark 2.10 in \cite{VeraarYaroslavtsev:2016}). We obtain a very direct proof of the identity
    $$
    V[F;a,x] = \int_a^x \abs{f(t)}dt,\qquad \textup{for } F(x) := \int_a^x f(t)dt.
    $$
\end{example}

\noindent \textbf{Acknowledgements}  This work was partially supported by The University of Costa Rica through the grant ``C3019- C\'{a}lculo Estoc\'{a}stico en Dimensi\'{o}n Infinita''. 




\section*{Declarations}

\noindent \textbf{Conflict of interest} The authors have no conflicts of interest to declare that are relevant to the content of this article.

\bibliographystyle{plain}
\bibliography{rieszlattices}

\end{document}